\documentclass[leqno,10pt, a4paper]{amsart}
\usepackage[utf8]{inputenc}
\usepackage{booktabs}
\usepackage{enumitem}
\usepackage{color}
\usepackage{amsmath, amsthm, amssymb}
\usepackage{graphicx}
\usepackage{mathrsfs}
\usepackage{url}
\usepackage{hyperref}
\newcommand{\K}{\ensuremath{\mathbb{K}}}
\newcommand{\C}{\ensuremath{\mathbb{C}}}

\newcommand{\N}{\ensuremath{\mathbb{N}}}

\newcommand{\Z}{\ensuremath{\mathbb{Z}}}
\newcommand{\CP}{\ensuremath{\mathbb{CP}}}
\newcommand{\coh}{\ensuremath{\mathrm{H}}}

\newcommand{\SU}{\ensuremath{\mathrm{SU}}}
\newcommand{\T}{\ensuremath{\mathrm{T}}}
\newcommand{\Sph}{\ensuremath{\mathrm{S}}}

\newcommand{\tensor}{\ensuremath{\otimes}}
\newcommand{\skpr}[2]{\left\langle #1, #2 \right\rangle}
\newcommand{\norm}[1]{\left\|#1\right\|}
\newcommand{\lie}[1]{\mathfrak{#1}}
\newcommand{\ad}{\ensuremath{\mathrm{ad}}}
 
\DeclareMathOperator{\Ric}{Ric}
\DeclareMathOperator{\diam}{diam} 

\newtheoremstyle{thm}
{0.6cm}
{0.6cm}
{\itshape}
{}
{\bfseries}
{.}
{0.5em}
{}
\theoremstyle{thm}

\newtheorem{thm}{Theorem}	
\newtheorem{prop}[thm]{Proposition}
\newtheorem{lem}[thm]{Lemma}
\newtheorem{question}[thm]{Question}

\newtheorem*{defi*}{Definition}

\numberwithin{equation}{section}
\let\stdthebibliography\thebibliography
\let\stdendthebibliography\endthebibliography
\renewenvironment*{thebibliography}[1]{%
  \stdthebibliography{Her14b}}
  {\stdendthebibliography}

\title[Homogeneous Spaces, Curvature and Cohomology]{Homogeneous Spaces, Curvature and Cohomology}
\author{Martin Herrmann}
\email{martin.herrmann@uni-muenster.de}
\address{Martin Herrmann \\Fakult\"at f\"ur Mathematik \\Karlsruher Institut f\"ur Technologie \\Kaiserstra\ss{}e 89--93 \\76133 Karlsruhe, Germany}
\curraddr{Martin Herrmann \\ Westf\"alische Wilhelms-Universit\"at M\"unster\\Mathematisches Institut \\Einsteinstra\ss{}e 62 \\48149 Münster, Germany}

\begin{document}
\begin{abstract}
We give new counterexamples to a question of Karsten Grove, whether there are only finitely many rational homotopy types among simply connected manifolds satisfying the assumptions of Gromov's Betti number theorem. Our counterexamples are homogeneous Riemannian manifolds, in contrast to previous ones. They consist of two families in dimensions 13 and 22. Both families are nonnegatively curved with an additional upper curvature bound and differ already by the ring structure of their cohomology rings with complex coefficients. The 22-dimensional examples also admit almost nonnegative curvature operator with respect to homogeneous metrics. 
\end{abstract}
\maketitle

\section{Introduction}
In this paper we give new counterexamples to a question raised by Grove  \cite{Grove1993} whether there are only finitely many rational  homotopy types of closed, simply connected manifolds with a given  lower bound  for the sectional curvature and a given upper bound on the diameter. Note that if in this question one replaces rational homotopy type by (integral) homotopy type, then the Aloff-Wallach spaces \cite{Aloff-Wallach} give rise to counterexamples.

Counterexamples to Grove's questions have been constructed by Fang and Rong  \cite{FangRong} and Totaro  \cite{Totaro}. All these examples actually already differ by the ring structure of their rational cohomology rings and, except for Totaro's six dimensional family, each family has  a uniform upper bound on the sectional curvature. Totaro's examples in dimension 6 and 9 have the merit of being nonnegatively curved, while the examples of Fang and Rong differ by their  cohomology rings with complex coefficients. 

We show that there are counterexamples to Grove's question with an additional upper curvature bound among nonnegatively curved homogeneous spaces in dimensions 13 and $\geq 15$, which differ by their complex cohomology rings. We also show that in dimensions $22$ and $\geq 24$ there are examples which in addition have almost nonnegative curvature operator with respect to homogeneous metrics.

Note that of Totaro's examples only finitely many can be homogeneous, as one can see from the classification of closed, simply connected, homogeneous spaces  up to dimension 9 carried out by Klaus in his diploma thesis \cite{Klaus}. In the nine-dimensional case, from the classification of Klaus one sees that, in order to be homogeneous, Totaro's nine-dimensional examples would have to be principal circle bundles over the product $(\Sph^2)^4=\Sph^2\times \Sph^2 \times \Sph^2\times \Sph^2$. These fall into  finitely many rational homotopy types.

\begin{thm}\label{MainTheorem13dimensional}
For $n=13$ and $n\geq 15$ there exist $C,D>0$ and an infinite family of closed, simply connected, homogeneous $n$--manifolds $M_a$ with normal homogeneous metrics, such that their cohomology rings $\coh^*(M_a;\C)$ with complex coefficients are pairwise non-isomorphic,  their sectional curvature satisfies  $0\leq \sec_{M_a}\leq C$ and  their diameter $\diam(M_a)\leq D$.
\end{thm}

Bounds on the curvature operator of a manifold are much stronger than bounds on the sectional curvature. For instance, the only closed, simply connected manifolds with positive curvature operator are spheres by B\"ohm and Wilking \cite{BW08}. This also finished the classification of simply connected manifolds with nonnegative curvature, see e.g.\ \cite{ChowLuNi06}. In particular, in a fixed dimension there are only finitely many diffeomorphism types of closed, simply connected manifolds with nonnegative curvature operator and Grove's question has a positive answer, if one replaces the lower sectional curvature bound by non-negativity of the curvature operator.

Recall that a closed manifold $M$ has \emph{almost nonnegative curvature operator} if for every $\varepsilon>0$ there is a Riemannian metric $g_\varepsilon$ on $M$ whose curvature operator $\hat{R}_{g_\varepsilon}$ satisfies
\[\hat{R}_{g_\varepsilon} (\diam(M,g_\varepsilon))^2>-\varepsilon.\]
We say that for a family $(g_t)_{t\in (0,1]}$ on $M$ the family $(M,g_t)$ has almost nonnegative curvature operator as $t\to 0$, if $\lim_{t\to 0}\lambda_\text{min}(\hat{R}_{g_t})(\diam(M,g_t))^2\geq 0$, where $\lambda_\text{min}(\hat{R}_{g_t})$ is the smallest eigenvalue of $\hat{R}_{g_t}$.  

In \cite{HST13} first examples of closed, simply connected manifolds with almost nonnegative curvature that do not admit a metric of nonnegative curvature operator were constructed. Furthermore, also in \cite{HST13} it was shown that there is an infinites family of closed, simply connected 6-manifolds which admit almost nonnegative curvature operator and have distinct (integral) homotopy types.

By looking at principal torus bundles over $(\CP^2)^5$ and their products with spheres we prove the following theorem.
\begin{thm}\label{MainTheoremANCO}
For $n=22$ and $n\geq 24$ there exist $C,D>0$ and infinitely many homogeneous, closed, simply connected $n$--manifolds $E_\alpha$ with pairwise non-isomorphic complex cohomology rings, such that on every $E_\alpha$ there exists a family $g_t$, $t\in (0,1]$, of homogeneous Riemannian metrics with
\begin{itemize}
\item$0\leq \sec_{g_1}\leq C$, $\Ric_{g_1}>0$ and $\diam_{g_1}\leq D$;
\item$0\leq \sec_{g_t}$, $\Ric_{g_t}>0$ and $\diam_{g_t}\leq D$ for every $t\in (0,1]$ and $(E_\alpha,g_t)$ has almost nonnegative curvature operator as $t \to 0$. 
\end{itemize}\end{thm}
The metric $g_1$ is in fact normal homogeneous with respect to the transitive $\SU(3)^5\times \T^2$ action on the $E_\alpha$. To our knowledge, this family gives the first examples of closed, simply connected manifolds in a fixed dimension with almost nonnegative curvature operator and infinitely many different rational homotopy types.

Note that in both theorems, the higher dimensional families are obtained from the families in dimensions 13 and 22 by taking products with spheres. By taking products with a circle, one can also obtain families in dimensions 14 and 23, respectively, that share the same properties, except for being simply connected. 

This paper is organized as follows. In Section \ref{SectionCohomologyRings} we will give the definition of the 22-dimensional examples and compute their cohomology rings. In Section \ref{SectionMetrics} we define the family of homogeneous metrics on the examples and estimate their curvature.  In Section \ref{SectionProof} we give the proof  of Theorem \ref{MainTheoremANCO}.  In Section \ref{Section13dimensional} we prove Theorem \ref{MainTheorem13dimensional}. Section \ref{SectionQuestions} contains some further remarks and questions.

Theorem \ref{MainTheoremANCO} was part of the author's dissertation \cite{MHDiss} at the Karls\-ruhe Institute of Technology. The author wishes to thank his advisor Wilderich Tuschmann for his constant help and advice and Anand Dessai for helpful discussions.

\section{Cohomology rings of the 22-dimensional examples}\label{SectionCohomologyRings}
In this section let $\K$ be a field of characteristic 0.

For nonzero $\alpha \in \N$ let $E_\alpha$ denote the principal $\T^2$-bundle obtained by first taking the principal $\Sph^1$-bundle $\tilde{E}$ over $(\CP^2)^5$ with first Chern class $z_1+z_2+z_3+z_4$ where $z_i$ is a generator of the integral cohomology of the $i$-th $\CP^2$ factor, and then the principal $\Sph^1$-bundle over $\tilde{E}$ with first Chern class $[z_1]+\alpha[z_3]+[z_5]$. Here $[z_i]$ denotes the equivalence class of $z_i$ in $\coh^2(\tilde{E};\Z)=\coh^2((\CP^2)^5;\Z)/(z_1+z_2+z_3+z_4)$.

Note that $E_\alpha$ is simply connected, since the Chern classes are no multiples in $\coh^2((\CP^2)^5;\Z)$ and $\coh^2(\tilde{E};\Z)$, respectively.

Let $x_i$ be the equivalence class of $z_i$ in $\coh^2(E_\alpha;\Z)$.  
\begin{lem}\label{Lemma22dimensionalCohomologyRings}
The $\K$-cohomology ring $\coh^{\leq 6}(E_\alpha;\K)$ of $E_\alpha$ in degrees $\leq 6$ is generated by $x_1$, $x_2$ and $x_3$ with relations
\[x_1^3,\; x_2^3,\; x_3^3,\]
\[ x_1^2x_2+ x_1^2 x_3 + x_1 x_2^2+2 x_1x_2x_3+x_1 x_3^2 +x_2^2x_3+x_2x_3^2,\]
and
\[  x_1^2x_3+\alpha\; x_1x_3^2.\]
\end{lem}

\begin{proof}
Let $\mu_{c_1(\tilde{E})}^k:\coh^k((\CP^2)^5;\K) \to\coh^{k+2}((\CP^2)^5;\K)$ denote multiplication by $c_1(\tilde{E})=z_1+z_2+z_3+z_4$. 

Look at the Serre spectral sequence of $\tilde{E}$. Since $\coh^{\mathrm{odd}}((\CP^2)^5;\K)=0$, the only nontrivial differentials are \[d_2:\coh^{2k}((\CP^2)^5;\K)\tensor \coh^1(\Sph^1;\K)\to \coh^{2k+2}((\CP^2)^5;\K)\tensor \coh^0(\Sph^1;\K).\] 
There is a generator $a\in \coh^1(\Sph^1;\K)$ such that  $d_2(a)=c_1(\tilde{E})$. Therefore  
\[\coh^{2k}(\tilde{E};\K)=\coh^{2k}((\CP^2)^5;\K)/\mathrm{Im}(\mu^{2k-2}_{c_1(\tilde{E})})\;\tensor\; \coh^0(\Sph^1;\K)\]
and
\[\coh^{2k+1}(\tilde{E};\K)=\ker(\mu^{2k}_{c_1(\tilde{E})})\tensor \coh^1(\Sph^1;\K).\]

The multiplicative structure is  induced from the  one on $\coh^*((\CP^2)^5;\K)\tensor \coh^*(\Sph^1;\K)$.

A short computation shows that $\mu_{c_1(\tilde{E})}^0$, $\mu_{c_1(\tilde{E})}^2$, $\mu_{c_1(\tilde{E})}^4$ and $\mu_{c_1(\tilde{E})}^6$ are injective. Furthermore, $\mathrm{Im}(\mu_{c_1(\tilde{E})}^0)$, $\mathrm{Im}(\mu_{c_1(\tilde{E})}^2)$, $\mathrm{Im}(\mu_{c_1(\tilde{E})}^4)$ and $\mathrm{Im}(\mu_{c_1(\tilde{E})}^6)$ identify products that contain $z_4$ with linear combinations of products that do not contain it. These identifications form a basis of $\mathrm{Im}(\mu_{c_1(\tilde{E})}^0)$ and $\mathrm{Im}(\mu_{c_1(\tilde{E})}^2)$. Since $z_4^3=0$, in $\mathrm{Im}(\mu_{c_1(\tilde{E})}^4)$ we get an additional relation $r$, namely
\begin{align*}
r&=c_1(\tilde{E})\cdot\tfrac{1}{3}(z_1^2+2 z_1 z_2+2 z_1 z_3-z_1 z_4+z_2^2+2 z_2 z_3-z_2 z_4+z_3^2-z_3 z_4+z_4^2)\\
&=z_1^2 z_2 + z_1^2z_3+z_1z_2^2+2 z_1z_2z_3+z_1 z_3^2+z_2^2z_3+z_2 z_3^2.
\end{align*}

A basis for $\mathrm{Im}(\mu_{c_1(\tilde{E})}^6)$ is given by the elements identifying non-vanishing products that contain $z_4$ with linear combinations of products that do not contain it and the  additional elements $r(z_1+z_2-z_3)$, $r(z_1-z_2+z_3)$ and $r(-z_1+z_2+z_3)$.

As the odd cohomology of $\tilde{E}$ vanishes in small degrees, the same works in degrees $\leq 6$ for $E_\alpha$ viewed as a principal $\Sph^1$-bundle over $\tilde{E}$. If $y_i$ denotes the equivalence class of $z_i$ in $\coh^2(\tilde{E};\K)$, then as above we get that the equivalence classes of $y_1$, $y_2$ and $y_3$ generate $\coh^{\leq 6}(E_\alpha;\K)$ and the additional relation is obtained as

\[ c_1(E_\alpha)\cdot\left(\frac{y_1^2}{\alpha }+2 y_1 y_3-\frac{y_1 y_5}{\alpha }+\alpha  y_3^2-y_3 y_5+\frac{y_5^2}{\alpha }\right)
=3( y_1^2 y_3 + \alpha y_1y_3^2).\qedhere\]
\end{proof}

\begin{lem}\label{CohomologyRingsDifferent}
For $\alpha, \tilde{\alpha}\in \N$ we have that $\coh^{*}(E_\alpha;\K)$ and $\coh^{*}(E_{\tilde{\alpha}};\K)$ are isomorphic if and only if $\alpha=\tilde{\alpha}$.
\end{lem}
\begin{proof}
Let $\varphi: \coh^*(E_\alpha;\K) \to \coh^*(E_{\tilde{\alpha}};\K)$ be an isomorphism of graded rings. Let $\varphi(x_i)=a_i x_1 + b_i x_2 + c_ix_3$. Then 
\begin{align*}0=\tfrac{1}{3}\varphi(x_i)^3&=(a_i^2 b_i - b_i c_i^2) x_1^2 x_2+ (a_ib_i^2-b_ic_i^2)x_1x_2^2+ 2(a_ib_ic_i-b_ic_i^2)x_1x_2x_3\\
&\qquad +(a_ic_i^2-b_ic_i^2-\tilde{\alpha}(a_i^2c_i-b_ic_i^2)) x_1x_3^2 + (b_i^2c_i-b_ic_i^2)x_2^2x_3.
\end{align*}
As $x_1^2 x_2$, $x_1x_2^2$, $x_1x_2x_3$, $x_1x_3^2$ and $x_2^2x_3$ form a basis of $\coh^6(E_{\tilde{\alpha}};\K)$, the coefficients must vanish. Therefore
\begin{align*}&\varphi(x_i)=a_i x_1,\quad \varphi(x_i)=b_i x_2,\quad \varphi(x_i)=c_i x_3, \quad\varphi(x_i)=b_i(x_1+x_2+x_3)\\&\text{or } \varphi(x_i)=a_i\,x_1 + \tilde{\alpha} \, a_i\, x_3.\end{align*}
Furthermore
\begin{equation}0=\varphi(x_1^2 x_3 + \alpha x_1 x_3^2)=\varphi(x_1)^2 \varphi(x_3) + \alpha\; \varphi(x_1) \varphi(x_3)^2.\label{eqnlastrelation}\end{equation}
Using this it is easy to see that $\varphi(x_1)=a_1 x_1$, $\varphi(x_1)=c_1 x_3$ or $ \varphi(x_1)=a_1\,x_1 + \tilde{\alpha} \, a_1\, x_3$ and  $\varphi(x_3)=a_3 x_1$, $\varphi(x_3)=c_3 x_3$ or $ \varphi(x_3)=a_3\,x_1 + \tilde{\alpha} \, a_3\, x_3$.

Then the matrix representing $\varphi|_{\coh^2(E_\alpha;\K)}:\coh^2(E_\alpha;\K)\to \coh^2(E_{\tilde{\alpha}};\K)$ takes one of the following forms:
\begin{align*} 
&(\mathrm{A}): \left(\begin{array}{c@{\hspace{6pt}}c@{\hspace{6pt}}c}a_1&&0\\0&?&0\\0&&c_3\end{array}\right),\quad &
&(\mathrm{B}): \left(\begin{array}{c@{\hspace{6pt}}c@{\hspace{4pt}}c@{\hspace{4pt}}}a_1&&a_3\\0&?&0\\0&&\tilde{\alpha}a_3\end{array}\right),\quad&
 &(\mathrm{C}):\left(\begin{array}{c@{\hspace{6pt}}c@{\hspace{6pt}}c}0&&a_3\\0&?&0\\c_1&&0\end{array}\right),\\
& (\mathrm{D}):  \left(\begin{array}{c@{\hspace{6pt}}c@{\hspace{4pt}}c@{\hspace{4pt}}}0&&a_3\\0&?&0\\c_1&&\tilde{\alpha} a_3\end{array}\right),\quad &
&(\mathrm{E}):  \left(\begin{array}{@{\hspace{3pt}}c@{\hspace{4pt}}c@{\hspace{6pt}}c}a_1&&a_3\\0&?&0\\\tilde{\alpha}a_1&&0\end{array}\right), \quad&
&(\mathrm{F}): \left(\begin{array}{@{\hspace{3pt}}c@{\hspace{4pt}}c@{\hspace{6pt}}c}a_1&&0\\0&?&0\\\tilde{\alpha}a_1&&c_3\end{array}\right).\end{align*}
Using (\ref{eqnlastrelation}) in these cases we get that
\[(\mathrm{A}): c_3=\frac{\tilde{\alpha}}{\alpha} a_1, \quad (\mathrm{B}): a_1=-\alpha a_3, \quad (\mathrm{C}): c_1=\alpha\tilde{\alpha} a_3,\]
\[(\mathrm{D}): c_1=- \alpha \tilde{\alpha} a_3, \quad (\mathrm{E}):a_1=- \alpha a_3,\quad(\mathrm{F}):a_1=-\frac{\alpha}{\tilde{\alpha}} c_3.\]

As $\varphi$ is an isomorphism, in each of these cases $\varphi(x_2)=b_2 x_2$ or $\varphi(x_2)=b_2(x_1+x_2+x_3)$. This gives us cases (A1), (A2), (B1), \dots, (F2).

Then use 
\begin{align}0&=\varphi(x_1^2 x_2 + x_1^2 x_3 + x_1 x_2^2 + 2 x_1 x_2 x_3 + x_1 x_3^2+ x_2^2 x_3 + x_2 x_3^2)\\
&=\varphi(x_1)^2 \varphi(x_2) + \varphi(x_1)^2 \varphi(x_3) + \varphi(x_1) \varphi(x_2)^2 + 2 \varphi(x_1) \varphi(x_2) \varphi(x_3)\nonumber\\
&\qquad + \varphi(x_1) \varphi(x_3)^2+ \varphi(x_2)^2 \varphi(x_3) + \varphi(x_2) \varphi(x_3)^2\nonumber
\end{align}
to get a contradiction in each of the cases (A1),\dots,(F2) if $\alpha\neq \tilde{\alpha}$. This is an elementary, but due to the number of cases rather lengthy, computation.
\end{proof}

\section{Homogeneous metrics on the 22-dimensional examples}\label{SectionMetrics}
Since $\tilde{G}=(\SU(3))^5$ is semisimple, the usual action of $\tilde{G}$ on $(\CP^2)^5$ lifts to an action on $E_\alpha$  that commutes with the $\T^2$ action (see \cite{Stewart}). As $\tilde{G}$ and $\T^2$ act transitively on base and fibre, respectively, $E_\alpha$ is a homogeneous space of $G=(\SU(3))^5 \times \T^2$. 

A short computation shows that the stabilizer $G_x$ at $x \in E_\alpha$ is of the form $\left\{(x,\rho(x))\middle|x \in \tilde{G}_{\pi(x)}\right\}$, where $\pi:E_\alpha \to (\CP^2)^5$ is the projection and $\rho:\tilde{G}_{\pi(x)} \to \T^2$ is a Lie group homomorphism. In particular, if we put $H=\mathrm{S}(\mathrm{U}(1)\times  \mathrm{U}(2))^5$ and $H_\rho=\left\{(x,\rho(x))\middle|x \in H\right\}$ for a homomorphism $\rho:H \to \T^2$, then $E_\alpha=G/H_\rho$ for some $\rho$. Furthermore, let  $K=(\mathrm{S}(\mathrm{U}(1)\times \mathrm{U}(2)))^5\times \T^2<G$.

Let $\skpr{.}{.}=\skpr{.}{.}_{\tilde{G}} \times \skpr{.}{.}_{\T^2}$ be the biinvariant product metric on $G$, where $\skpr{.}{.}_{\tilde{G}}$ and $\skpr{.}{.}_{\T^2}$ are biinvariant metrics on $\tilde{G}$ and $\T^2$. We denote the Lie algebras of $G$, $K$ and $H_\rho$ with  $\lie{g}$, $\lie{k}$ and $\lie{h}_\rho$, respectively. Let $\lie{m}_1$ be an orthogonal complement of $\lie{k}$ in $\lie{g}$ and $\lie{m}_{2,\rho}$ an orthogonal complement of $\lie{h}_\rho$ in $\lie{k}$. Then define a left invariant metric, which is invariant under the adjoint actions of $H_\rho$ and $K$ on $\lie{g}$, by 
\[\skpr{X}{Y}_t=\skpr{X_{\lie{m}_1}}{Y_{\lie{m}_1}}+t^2 \skpr{X_{\lie{m}_{2,\rho}}}{Y_{\lie{m}_{2,\rho}}}+t^2\skpr{X_{\lie{h}_\rho}}{Y_{\lie{h}_\rho}},\]
where $X,Y\in \lie{g}$ and the subscripts denote orthogonal projection on the respective subspaces. The left invariant metric $\skpr{.}{.}_t$ induces a homogeneous metric $g_t$ on $E_\alpha$. 

The following lemma will allow us to use the main result of \cite{HST13} to conclude that $(M,g_t)$ has almost nonnegative curvature operator as $t\to 0$. 
\begin{lem}\label{metricsareconnectionmetrics}
The homogeneous metrics $g_t$ are connection metrics on $E_\alpha$ with respect to some connection form $\gamma$ and some biinvariant metric $b$ on $\T^2$.
\end{lem}
\begin{proof}
Let $\iota_g=(p\circ l_g)_{*_e}:\lie{m}_1\oplus \lie{m}_{2,\rho} \to \T_{gH_\rho} E_\alpha$, where $p:G \to G/H_\rho=E_\alpha $ is the canonical projection and $l_g$ is left multiplication by $g$. Then $\iota_g$ is an isomorphism.

Note that the vertical subspaces at $gH_\rho\in E_\alpha$ is given by $\iota_g( \lie{m}_{2,\rho})$. Define horizontal subspaces $\mathcal{H}_{gH_\rho}=\iota_g( \lie{m}_1)$ . Then ${R_t}_*\mathcal{H}_{gH_\rho}=\mathcal{H}_{R_t(gH_\rho)}$ for $t\in \T^2$ and the right action $R$ of $\T^2$ on $E_\alpha$, so $\mathcal{H}$ defines a connection form $\gamma$ on $E_\alpha$. Let $\rho_*^*$ be the adjoint of $\rho_*:\lie{s}(\lie{u}(1)\times \lie{u}(2)) \to \lie{t}^2$ with respect to the restriction of $\skpr{.}{.}_{\tilde{G}}$ to $\mathrm{S}(\mathrm{U}(1)\times \mathrm{U}(2))$ and $\skpr{.}{.}_{\T^2}$. Then \[\lie{m}_{2,\rho}=\{(-\rho_*^*v,v)| v \in \lie{t}^2\}\] and 
\[\iota_g(-{\rho_*}^*v,v)=\frac{d}{dt}|_{t=0}\left(gH_\rho.\exp( t (v+\rho_* {\rho_*}^*v))\right).\]
Therefore $\gamma=\tilde{\gamma}\circ\iota_g^{-1} $ with \[\tilde{\gamma}:\lie{m}_1\oplus \lie{m}_{2,\rho} \to \lie{t}^2, \qquad X+(-{\rho_*}^*v,v)\mapsto v+\rho_*{\rho_*}^*v.\] 
Define a biinvariant metric on $\lie{t}^2$ by $b=(\tilde{\gamma}|_{\lie{m}_{2,\rho}}^{-1})^*\skpr{.}{.}$. Then $g_t$  is the connection metric induced by $\gamma$, $t^2 b$ and the normal homogeneous metric on $G/K=(\CP^2)^5$ coming from $\skpr{.}{.}$.
\end{proof}

\section{Curvature of the metrics}
Let $X^t=X_1+\tfrac{1}{t} X_2, Y^t=Y_1 +\tfrac{1}{t} Y_2\in \lie{m}_1\oplus \lie{m}_{2,\rho}$ be orthonormal with respect to $\skpr{.}{.}_t$. Then the sectional curvature with respect to the homogeneous metric $g_t$ is given by (see e.g. \cite[formula (3.32)]{CheegerEbin})
\begin{align*}
\sec_{g_t}(p_* X^t,p_*Y^t)&=\norm{\ad_{X^t}^{*_t}Y^t + \ad_{Y^t}^{*_t}X^t}_t^2-\skpr{\ad_{X^t}^{*_t}X^t}{\ad_{Y^t}^{*_t}Y^t}\\
&\qquad -\frac{3}{4}\norm{[X^t,Y^t]_{\lie{m}_1 \oplus \lie{m}_{2,\rho}}}_t^2\\
&\qquad-\frac{1}{2}\left( \skpr{[[X^t,Y^t],Y^t]}{X^t}_t+\skpr{[[Y^t,X^t],X^t]}{Y^t}_t \right).
\end{align*}

Here $\ad_X^{*_t}$ denotes the adjoint of $\ad_X:\lie{g} \to \lie{g} , \; Y \mapsto [X,Y]$ with respect to $\skpr{.}{.}_t$.

If, for $r>0$, we let $f_r:\lie{g} \to \lie{g}, \; X_{\lie{m}_1}+X_{\lie{m}_{2,\rho}}+X_{\lie{h}_\rho} \mapsto X_{\lie{m}_1}+r (X_{\lie{m}_{2,\rho}}+ X_{\lie{h}_\rho})$, then $\ad_X^{*_t}=f_{\frac{1}{t^2}}\circ \ad_X^{*_1} \circ f_{t^2}$ and $\ad_X^{*_1}=-\ad_X$. The Lie brackets satisfy
$[\lie{m}_1,\lie{m}_{2,\rho}]\subseteq[\lie{m}_1,\lie{k}]\subseteq \lie{m}_1$, $[\lie{k},\lie{k}] \subseteq \lie{k}$, $[\lie{m}_{2,\rho},\lie{m}_{2,\rho}] =\{0\}$ and $[\lie{m}_1,\lie{m}_1] \subseteq\lie{k}$.

Using this, we compute the expressions appearing in the formula for the sectional curvature of  the plane spanned by $X^t$ and $Y^t$.
\begin{lem}With $X^t$ and $Y^t$ as above:\label{differenttermscurvature}
\begin{enumerate}[label=\emph{(\alph*)}]
\item
$\begin{aligned}[t]&\!\norm{\ad_{X^t}^{*_t}Y^t + \ad_{Y^t}^{*_t}X^t}_t^2\\&\qquad=(t-\tfrac{1}{t})^2\Big(\norm{[X_1,Y_2]}^2 + \norm{[X_2,Y_1]}^2-2\skpr{[X_1,Y_2]}{[X_2,Y_1]}\Big)\\\end{aligned}$
\item $\skpr{\ad_{X^t}^{*_t}X^t}{\ad_{Y^t}^{*_t}Y^t}_t=-(t-\frac{1}{t})^2\skpr{[X_1,Y_2]}{[X_2,Y_1]}$
\item$\begin{aligned}[t]&\!\norm{[X^t,Y^t]_{\lie{m}_1 \oplus \lie{m}_{2,\rho}}}_t^2\\&\qquad=t^2\norm{[X_1,Y_1]_{\lie{m}_{2,\rho}}}^2\\&\qquad \qquad+\frac{1}{t^2}\left(\norm{[X_1,Y_2]}^2+\norm{[X_2,Y_1]}^2 + 2 \skpr{[X_1,Y_2]}{[X_2,Y_1]}\right)\end{aligned}$
\item $\begin{aligned}[t]&\!-\frac{1}{2}\left( \skpr{[[X^t,Y^t],Y^t]}{X^t}_t+\skpr{[[Y^t,X^t],X^t]}{Y^t}_t \right)\\&\qquad=\norm{[X_1,Y_1]}^2\\&\qquad\qquad+\frac{1+\frac{1}{t^2}}{2}\left(\norm{[X_1,Y_2]}^2+\norm{[X_2,Y_1]}^2 + 2 \skpr{[X_1,Y_2]}{[X_2,Y_1]}\right)\\
\end{aligned}$
\end{enumerate}
\end{lem}

\begin{prop}\label{propNonnegativeSecPositiveRic}
$(E_\alpha,g_t)$ has nonnegative sectional and positive Ricci curvature for all $t\in (0,1]$.
\end{prop}
\begin{proof}
Using Lemma \ref{differenttermscurvature} one  computes
\begin{align*}\sec_{g_t}(p_*X^t,p_*Y^t)&=\norm{[X_1,Y_1]}^2-\tfrac{3}{4}t^2 \norm{[X_1,Y_1]_{\lie{m}_{2.\rho}}}^2\\
&\qquad+\left(t^2+\tfrac{3}{4 t^2}-\tfrac{3}{2}\right)\norm{[X_1,Y_2]-[X_2,Y_1]}^2\\
&\qquad +t^2 \skpr{[X_1,Y_2]}{[X_2,Y_1]}.
\end{align*}
For $t=1$ this expression is nonnegative, since $g_1$ is normal homogeneous. We get
\[\skpr{[X_1,Y_2]}{[X_2,Y_1]}\geq -\norm{[X_1,Y_1]}^2+\tfrac{3}{4}\norm{[X_1,Y_1]_{\lie{m}_{2.\rho}}}^2-\tfrac{1}{4}\norm{[X_1,Y_2]-[X_2,Y_1]}^2\]
and therefore 
\begin{align*}\sec_{g_t}(p_* X^t,p_*Y^t)&\geq (1-t^2)\norm{[X_1,Y_1]}^2 + \left(\tfrac{3}{4}(t^2+\tfrac{1}{ t^2})-\tfrac{3}{2}\right)\norm{[X_1,Y_2]-[X_2,Y_1]}^2\\
&\geq 0.\end{align*}

To see that $(E_\alpha, g_t)$ has positive Ricci curvature, use that the normal homogeneous metric $g_1$ has positive Ricci curvature by Berestovskij \cite{Beres95} and apply the same trick as above.

\end{proof}

\section{Proof of Theorem \ref{MainTheoremANCO}}\label{SectionProof}
To get the conclusion about the metrics $g_t$ in dimension 22 we just combine Lemma \ref{CohomologyRingsDifferent} and Proposition \ref{propNonnegativeSecPositiveRic} and use the main result of \cite{HST13}, which we can use due to Lemma \ref{metricsareconnectionmetrics}, to get almost nonnegative curvature operator. For the diameter we have $\diam(E_\alpha,g_t)\leq \diam(G,\skpr{.}{.}_1)$.

To get the additional upper bound for the sectional curvature if $t=1$, note that 
\[\sec_{(M,g_1)}(p_*X,p_*Y)=\tfrac{1}{4}\norm{[X,Y]_{\lie{m}_1\oplus\lie{m}_{2,\rho}}}^2+\norm{[X,Y]_{\lie{h}_\rho}}^2,\]
so that
\[\sec_{(M,g_1)}\leq \max_{\substack{X,Y\in \lie{g}\\X,Y \text{ orthonormal}}}\norm{[X,Y]}^2.\]

If $n\geq3$, the same proof as above shows, that $\coh^*(E_\alpha\times \Sph^n;\K)$ and $\coh^*(E_{\tilde{\alpha}}\times \Sph^n;\K)$ are isomorphic if and only if $\alpha=\tilde{\alpha}$. For $n=2$ use that an isomorphism $\varphi:\coh^*(E_\alpha\times \Sph^2;\K)\to \coh^*(E_{\tilde{\alpha}}\times \Sph^2;\K)$ maps the generator of $z\in \coh^*(E_\alpha\times \Sph^2;\K)$ corresponding to $\Sph^2$ to a multiple of $z \in \coh^*(E_{\tilde{\alpha}}\times \Sph^2;\K)$, since this are the only elements whose squares vanish. For the other generators $x_i$ we have that $\varphi(x_i)$ is a linear combination of the $x_i$, since $\varphi(x_i)^3=0$ and $\varphi(x_i) \notin \K \cdot z$. 

\section{The 13-dimensional family}\label{Section13dimensional}

As in the  proof of Theorem \ref{MainTheoremANCO}, the case $n\geq 15$ in Theorem \ref{MainTheorem13dimensional} follows by considering products of the 13-dimensional family with spheres of appropriate dimension. The proof for the 13-dimensional family is similar, but easier than the one for the 22-dimensional family, since we do not have to deal with a family of metrics. The examples we consider are the following:

Let $B=(\SU(3)/\T^2)\times (\SU(3)/\T^2)$. Then by Borel \cite{Borel}
\[\coh^*(B;\Z)=\Z[x_1, y_1, x_2,y_2]/I,\]
where $I$ is the ideal generated by $x_1^2+x_1 y_1+y_1^2$, $x_1^2 y_1+x_1 y_1^2$, $x_2^2+x_2 y_2+y_2^2$ and $x_2^2 y_2+x_2 y_2^2$. For $a\in \Z$ let $M_a$ be the principal $\Sph^1$-bundle over $B$ with first Chern class $c_1(M_a)=a x_1+ y_1 -y_2$. Again by \cite{Stewart}, $M_a$ is a homogeneous space of $\SU(3)\times \SU(3)\times \Sph^1$.

Fix a biinvariant metric on $\SU(3)\times \SU(3)\times \Sph^1$. If we equip all $M_a$ with the induced metric, we get as in the proof of Theorem \ref{MainTheoremANCO} in Section \ref{SectionProof}, that $0\leq \sec_{M_a}\leq C$ for some constant independent of $a$ and $\diam M_a\leq \diam(\SU(3)\times \SU(3)\times \Sph^1)$.  Therefore, for the proof of Theorem \ref{MainTheorem13dimensional} we are left with showing that among the complex cohomology rings of the $M_a$, $ a \in \Z$, there are infinitely many which are pairwise non-isomorphic.

\begin{lem}\label{LemmaCohomologyRingMa}
The cohomology ring $\coh^{\leq 6}(M_a;\C)$ in degrees $\leq 6$ is generated by $x_1$, $y_1$ and $x_2$ with only one additional relation, namely $r_a=(a^2-1) x_1^2 +(2a-1)x_1y_1 +x_2^2 + a x_1 x_2 + y_1 x_2$.
\end{lem}
The proof of Lemma \ref{LemmaCohomologyRingMa} is similar to that of Lemma \ref{Lemma22dimensionalCohomologyRings} and we omit it here. Note that the cohomology ring $\coh^*(M_a;\C)$ is  determined by the ring structure of $\coh^{\leq 6}(M_a;\C)$ via Poincaré duality.
\begin{lem}
Let $\omega \in \coh^2(M_a;\C)$. Then $\omega^3=0$, if and only if  
\begin{itemize}
\item in the case $a\neq 0,1$, we have that for some $\lambda \in \C$: $\omega = \lambda x_1$, $\omega = \lambda y_1$ , $\omega = \lambda x_2$  or $\omega = \lambda (x_1+y_1)$
\item in the case $a=0$, we have that for some $\lambda \in \C$: $\omega = \lambda x_1$, $\omega = \lambda y_1$ , $\omega = \lambda x_2$ , $\omega = \lambda (x_1+y_1)$ or $\omega= \lambda (y_1+x_2)$;
\item in the case $a=1$, we have that for some $\lambda \in \C$: $\omega = \lambda x_1$, $\omega = \lambda y_1$ , $\omega = \lambda x_2$ , $\omega = \lambda (x_1+y_1)$ or $\omega= \lambda (x_1+y_1+x_2)$.
\end{itemize}
\end{lem}

\begin{proof}
Let $\omega= \alpha x_1 +\beta y_1 + \gamma x_2$. The coefficients of $\omega^3$ in the basis $x_1^2y_1$, $ x_1^2x_2$, $x_1 y_1 x_2$ of $\coh^6(M_a;\C)$ then give 
\begin{align}
0&= \alpha ^2 \beta - \alpha  \beta ^2+ \alpha  \gamma ^2-2 a \alpha  \gamma ^2+2 a \beta  \gamma ^2- a^2 \beta  \gamma ^2 \label{omegacubedequ1}\\
0&=\gamma ( \alpha ^2  - \beta ^2  - a \alpha  \gamma + \beta  \gamma)\label{omegacubedequ2}\\
0&= \gamma(2 \alpha  \beta   - \beta ^2  - \alpha  \gamma+ \beta  \gamma- a \beta  \gamma ).\label{omegacubedequ3}
\end{align}
If $\gamma=0$, then Equation \ref{omegacubedequ1} becomes $\alpha^2 \beta =\alpha\beta^2$, so $\alpha=0$, $\beta=0$ or $\alpha=\beta$.

If $\gamma\neq 0$, we can assume $\gamma=1$ since $\gamma^{-1} \omega$ also cubes to zero. 
Then Equation \ref{omegacubedequ3} gives $\alpha=\beta\frac{\beta-1+a}{2 \beta -1}$. Note that because of (\ref{omegacubedequ3}) and $a\in\Z$, we get $2\beta-1\neq 0$. Substituting this  for $\alpha$ in (\ref{omegacubedequ2}) gives 
\[0=-\beta (\beta-1)(3 \beta^2 -3 \beta +a^2-a+1).\]
If $\beta=0$, then also $\alpha =0$. If $\beta=1$, then $\alpha=a$ and (\ref{omegacubedequ1}) gives $a=0$ or $a=1$. These are exactly the possibilities for $\omega$ in the lemma.

The last case is, that $0=3 \beta^2 -3 \beta +a^2-a+1$. In this case (\ref{omegacubedequ1}) gives  equations for $a$, that do not have roots in the integers. 
\end{proof}
 The preceding lemma gives strong restriction on the form of a possible isomorphism and allows us to prove the following proposition.
\begin{prop}\label{Prop13dimensionalCohomologyRingsNonIsomorphic}
Let $a,b \in \Z$, $a,b\geq 2$. Suppose that $\varphi:\coh^*(M_b;\C) \to \coh^*(M_a;\C)$ is an isomorphism. Then $a=b$.
\end{prop}
\begin{proof}
Since $\varphi$ is a ring isomorphism, it satisfies  $\omega^3=0$ if and only if  $\varphi(\omega)^3=0$. Therefore, the  generators $x_1, y_1, x_2\in \coh^2(M_b;\C)$ are mapped to multiples of $x_1, y_1, x_2$ or $ x_1+y_1\in \coh^2(M_a;\C)$. Since $(\varphi(x_1)+\varphi(y_1))^3=0$, we know that $\varphi(x_1), \varphi(y_1) \in \C x_1 \oplus \C y_1 \subset \coh^2(M_a;\C)$. Therefore, $\varphi(x_2)=\mu x_2$ for some $\mu \in \C\setminus\{0\}$. There are six possibilities for the images of $x_1$ and $y_1$. The matrix representing $\varphi|_{\coh^2(M_b;\C)}$ in the basis $x_1$, $y_1$, $x_2$ is given by
\[\begin{pmatrix}\lambda&&\\&\lambda&\\&&\mu\\\end{pmatrix},
\begin{pmatrix}&\lambda&\\\lambda&&\\&&\mu\\\end{pmatrix},
\begin{pmatrix}\lambda &-\lambda&\\\lambda&&\\&&\mu\\ \end{pmatrix}, \]
\[\begin{pmatrix}\lambda&&\\\lambda&-\lambda&\\&&\mu\\ \end{pmatrix}, 
\begin{pmatrix}-\lambda&\lambda&\\&\lambda&\\&&\mu\\\end{pmatrix} \text{ or } 
\begin{pmatrix}&\lambda&\\-\lambda&\lambda&\\&&\mu\\\end{pmatrix}\]
for some $ \lambda \in \C\setminus \{0\}$. Furthermore, $\varphi$ has to satisfy 
\[0=\varphi(r_a)=\varphi\big((a^2-1) x_1^2 +(2a-1)x_1y_1 +x_2^2 + a x_1 x_2 + y_1 x_2\big).\]
Using the coefficients of $\varphi(r_a)$ in the basis $x_1^2$, $x_1 y_1$, $x_1x_2$, $y_1x_2$ of $\coh^4(M_b;\C)$, in each of the above cases one sees that either $a=b$  or there is no isomorphism of this form with $a,b\geq2$.
\end{proof}
Together with the discussion of the curvature of  $M_a$ Proposition \ref{Prop13dimensionalCohomologyRingsNonIsomorphic} concludes the proof of Theorem \ref{MainTheorem13dimensional}.

\section{Further remarks and questions}\label{SectionQuestions}
Totaro's six- and seven-dimensional examples are optimal in terms of the dimension. One might ask what is the optimal dimension if one restricts to homogeneous spaces.  
\begin{question}
What is lowest dimension in which homogeneous counterexamples to Grove's question (with an upper curvature bound) exist?
\end{question}

As mentioned in the introduction, Stephan Klaus has classified all closed, simply connected, homogeneous spaces in dimensions $\leq9$. In dimensions $\leq 6$ and in dimension 8 there are only finitely many diffeomorphism types of homogeneous spaces. Using other results from \cite{MHDiss} written up in \cite{MH14remild} one sees that in the missing  dimensions 7 and 9 there are only finitely many rational homotopy types of closed, simply connected, homogeneous spaces. Therefore the optimal dimension $n$ satisfies $10\leq n\leq 13$.

Relaxing the symmetry assumption one might also ask the same question for cohomogeneity one manifolds. Recently, Anand Dessai \cite{Dessai15} has found an infinite sequence of eight-dimensional cohomogeneity one manifolds of nonnegative curvature with pairwise non-isomorphic complex cohomology rings. This is the minimal dimension in which such cohomogeneity one examples can exist.  
\begin{question}
What is the lowest possible dimension in which nonnegatively curved, cohomogeneity one counterexamples  to Grove's question with an additional uniform upper curvature bound exist?\end{question}
 Here the optimal dimension $n$ satisfies $8\leq n \leq 13$. 

Dropping all symmetry assumptions one might still ask about the minimal dimension in which an infinite family of nonnegatively curved counterexamples with a uniform upper curvature and diameter bounds exist. By the examples of Totaro the minimal dimension is at most 9 and by a result of Tuschmann (see \cite{Tuschmann02}) greater or equal 7. Note however, that due to the classification of closed, simply connected, rationally elliptic manifolds in dimension 7 (see \cite{MH14remild}), closed, simply connected, nonnegatively curved 7-manifolds can only have finitely many different complex homotopy types, if the Bott conjecture holds.

\bibliographystyle{amsalpha}
\bibliography{refs}{}
\end{document}